\documentclass[12pt]{amsart}
\usepackage[pdfborder={0 0 0},colorlinks,linkcolor=black,citecolor=black,urlcolor=blue]{hyperref}
\usepackage{cleveref,caption,subfigure}
\usepackage{url}
\usepackage{pgf,tikz}

\usetikzlibrary{arrows}
\usetikzlibrary{scopes}
\usetikzlibrary{positioning}
\newcommand{\midarrow}{\tikz \draw[thick,-angle 60] (0,0) -- +(.1,0);}

\newtheorem{Thm}{Theorem}
\newtheorem*{Thm*}{Theorem}
\newtheorem{Prop}[Thm]{Proposition}
\newtheorem{Lemma}[Thm]{Lemma}
\newtheorem{Cor}[Thm]{Corollary}

\crefname{Thm}{Theorem}{Theorems}
\crefname{Lemma}{Lemma}{Lemmas}
\crefname{Prop}{Proposition}{Propositions}
\crefname{Cor}{Corollary}{Corollaries}
\crefname{section}{\S}{\S}
\crefname{figure}{Figure}{Figures}

\newcommand{\eps}{\epsilon}

\renewcommand{\P}{\mathbb P}

\newcommand{\df}{\textbf}

\title{The Phase Transition for Dyadic Tilings}
\date{11 July 2011 (revised 19 July 2012)}

\author[Angel]{Omer Angel}
\address[O.\ Angel]{\tiny Department of Mathematics, University of British
  Columbia, Vancouver, BC V6T 1Z2, Canada}
\email{angel at math.ubc.ca}
\urladdr{\url{http://math.ubc.ca/~angel/}}

\author[Holroyd]{Alexander E.\ Holroyd}
\address[A.\ E.\ Holroyd]{\tiny Microsoft Research,
1 Microsoft Way, Redmond, WA 98052, USA} \email{holroyd
at microsoft.com}
\urladdr{\url{http://research.microsoft.com/~holroyd/}}

\author[Kozma]{Gady Kozma}
\address[G.\ Kozma]{\tiny The Weizmann Institute of Science,
Rehovot POB 76100, Israel} \email{gady.kozma at weizmann.ac.il}
\urladdr{\url{http://www.wisdom.weizmann.ac.il/~gadyk/}}

\author[W\"astlund]{Johan W\"astlund}
\address[J.\ W\"astlund]{\tiny Department of Mathematical Sciences,
Chalmers University of Technology, S-412 96 Gothenburg, Sweden}
 \email{wastlund at chalmers.se}
\urladdr{\url{http://www.math.chalmers.se/~wastlund/}}

\author[Winkler]{Peter Winkler}
\address[P.\ Winkler]{\tiny Department of Mathematics, Dartmouth,
Hanover, NH 03755-3551, USA}
\email{peter.winkler at dartmouth.edu}
\urladdr{\url{http://www.math.dartmouth.edu/~pw/}}


\begin{document}
\begin{abstract}
A dyadic tile of order $n$ is any rectangle obtained from the unit square by
$n$ successive bisections by horizontal or vertical cuts.  Let each dyadic
tile of order $n$ be available with probability $p$, independently of the
others. We prove that for $p$ sufficiently close to $1$, there exists a set
of pairwise disjoint available tiles whose union is the unit square, with
probability tending to $1$ as $n\to\infty$, as conjectured by Joel Spencer in
1999.  In particular we prove that if $p=7/8$, such a tiling exists with
probability at least $1-(3/4)^n$.  The proof involves a surprisingly delicate
counting argument for sets of unavailable tiles that prevent tiling.
\end{abstract}
\maketitle

\section{Introduction} \label{S:intro}

A \df{dyadic tile} of \df{order} $n$ is a rectangle of the form
\[
\Bigl[\frac{a}{2^i}, \frac{a+1}{2^i}\Bigr] \times \Bigl[\frac{b}{2^j},
\frac{b+1}{2^j}\Bigr],
\]
where $a,b,i,j$ are integers and $i+j=n$. We consider only tiles that are
subsets of the unit square $[0,1]^2$, which is to say that $0\leq a <2^i$ and
$0\leq b < 2^j$. The tiles of order $n$ come in $n+1$ different \df{shapes},
each shape corresponding to a particular choice of $i$ and $j$. There are
$2^n$ tiles of each shape, and thus in total $(n+1)\,2^n$ tiles of order $n$.
A \df{tiling} of a rectangle $R$ is a set of tiles whose union is $R$ and
whose interiors are pairwise disjoint. \cref{F:exampleTiling} shows a tiling
of the unit square by tiles of order 3; for visual clarity we illustrate
tiles by rectangles with rounded corners, slightly smaller than their true
sizes.

\begin{figure}
\begin{minipage}{0.49\textwidth}
\centering
\includegraphics[width=4.5cm]{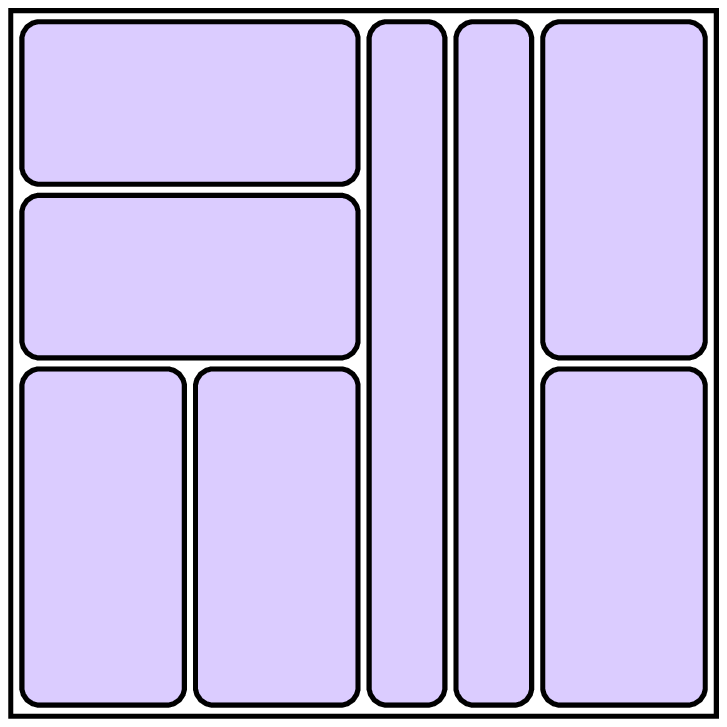}
\captionsetup{width=0.8\textwidth}
\caption{A dyadic tiling by order-3 tiles.}
\label{F:exampleTiling}
\end{minipage}
\hfill
\begin{minipage}{0.49\textwidth}
\centering
\begin{tikzpicture}[domain=0:1, scale=3.5, thick]
\draw[->] (0,0) -- (1.1,0) node[right] {$p$};
\draw[->] (0,0) -- (0,1.1);
\draw[color=blue] plot (\x,\x);
\draw[color=green!50!black, smooth] plot (\x, {2*\x^2-\x^4});
\draw[color=red, smooth] plot
(\x,{7*\x^4-8*\x^6-4*\x^7+\x^8+8*\x^9-4*\x^11+\x^12});
\draw[color=blue] (-0.2, 0.9) node {$T_0(p)$};
\draw[color=green!50!black] (-0.2, 0.75) node {$T_1(p)$};
\draw[color=red] (-0.2, 0.6) node {$T_2(p)$};
\end{tikzpicture}
\captionsetup{width=0.8\textwidth}
\caption{Tiling probabilities $T_n(p)$ for $n=0,1,2$.} \label{F:Tplots}
\end{minipage}
\end{figure}

Suppose that each tile of order $n$ is \df{available} with probability $p$
independently of the other tiles. Let $T_n(p)$ denote the probability that
there exists a set of available order-$n$ tiles that constitutes a tiling of
the unit square $[0,1]^2$. For example $T_0(p) = p$ trivially, and $T_1(p) =
2p^2-p^4$ since each of the vertical and horizontal tilings is available with
probability $p^2$ and both are available with probability $p^4$.  A more
involved calculation shows that $T_2(p) =
7p^4-8p^6-4p^7+p^8+8p^9-4p^{11}+p^{12}$ (the term $7p^4$ corresponds to the
$7$ distinct tilings by order-$2$ tiles).  The functions $T_0,T_1,T_2$ are plotted
in Figure~\ref{F:Tplots}.

It is natural to define the critical probability
\[
p_c:=\inf\big\{p:\lim_{n\to\infty} T_n(p)=1\big\}.
\]
Joel Spencer asked in 1999 whether $p_c<1$ (personal communication). The main
result of this paper is an affirmative answer to this question. In particular
we show the following.

\begin{Thm}\label{T:7/8}
We have
\begin{align*}
T_n(7/8) &\geq 1 - (3/4)^n;\\
T_n(6/7) &\geq 1 - (16/17)^n;\\
T_n(0.8560310279) &\geq 1 - (0.999998)^n.
\end{align*}
In particular, $p_c\leq 0.8560310279$.
\end{Thm}

We explain in \cref{analysis} where these numbers come from.   The first
inequality can be checked by hand, while the second and third involve
rigorous computer-assisted numerical methods.  The bound $p_c\leq
0.856\cdots$ is the best that can be obtained with our method, but we believe
that $p_c$ is smaller than this. Using standard sharp threshold technology,
we also establish the following.

\begin{Thm}\label{T:sharp}With $p_c$ defined as above,
\[
\lim_{n\to\infty} T_n(p) = \begin{cases} 0 \quad \text{if $p<p_c$,}\\ 1
\quad \text{if $p>p_c$}. \end{cases}
\]
\end{Thm}

A straightforward argument proves the lower bound $p_c\geq (\surd5{-}1)/2$
$=0.618\cdots$, and in fact this may be improved to
$$p_c\geq 0.785996.$$
It is likely that this bound could be improved still
further---see \cref{prelim} for more information.

In $d$ dimensions we may similarly define an order-$n$ dyadic tile to be a
product of $d$ dyadic intervals.  Let each dyadic tile of volume $2^{-n}$ be
available independently with probability $p$, and let $p_c(d)$ be the infimum
of $p$ for which there is a tiling of the cube $[0,1]^d$ with probability
tending to $1$ as $n\to\infty$.  It is immediate that $p_c(d)$ is
non-increasing in $d$ (since the product of any $d$-dimensional tiling with
$[0,1]$ is a $(d+1)$-dimensional tiling), so \cref{T:7/8} implies $p_c(d)<1$
for all $d\geq 2$.  In dimension $3$ a simple argument gives the lower bound
$p_c(3)\geq 1/8$ (see \cref{prelim}), but for $d\geq 4$ we do not know
whether $p_c(d)>0$.

A different model, of uniformly random dyadic tilings, was investigated by
Janson, Randall and Spencer \cite{J01, JRS02}.  Also see \cite{LSV02} for
enumeration of tilings, and \cite{CLSW01} for a related problem of random
packing.

\section{Preliminaries, and outline of proof}
\label{prelim}

The following key observation is Theorem~1.1 of \cite{JRS02}.

\begin{Lemma}\label{L:JRS}
A dyadic tiling by tiles of order $n\geq 1$ consists either of tilings of the two
horizontal rectangles $[0,1]\times [0,1/2]$ and $[0,1]\times [1/2,1]$, or
tilings of the two vertical rectangles $[0,1/2]\times [0,1]$ and
$[1/2,1]\times [0,1]$.
\end{Lemma}

\begin{proof}
  The only tiles that cross the median line $\{1/2\}\times[0,1]$ are those of the
  most horizontal shape, i.e.\ of the form $[0,1] \times [b/2^{n}, (b+1)/2^{n}]$.
  Similarly the only tiles that cross $[0,1]\times\{1/2\}$ are of the most
  vertical shape. There cannot be tiles of both these shapes in a tiling,
  since they intersect (see \cref{F:intersect}).
\end{proof}

We remark that the analogous statement to \cref{L:JRS} fails in dimensions
greater than $2$: see \cref{F:3d} for a counterexample in dimension $3$.

\begin{figure}
\begin{center}
\includegraphics[width=3.3cm]{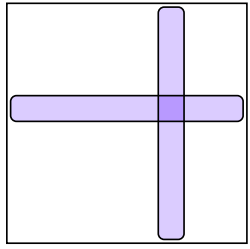}
\end{center}
\caption{A horizontal and a vertical tile intersect, and therefore
cannot both be present in a tiling.
}
\label{F:intersect}
\end{figure}
\begin{figure}
\begin{center}
\begin{tikzpicture} [scale=0.6,thick]
\draw (0,2) -- (2,0) -- (6,0) -- (6,4) -- (4,6) -- (0,6) -- cycle;
\draw (2,2) -- (4,2) -- (4,4) -- (3,5) -- (1,5) -- (1,3) -- cycle;
\draw (1,1) -- (1,3) -- (3,3);
\draw (2,6) -- (3,5) -- (3,3);
\draw (6,2) -- (4,2) -- (3,3);
\draw (2,0) -- (2,2);
\draw (0,6) -- (1,5);
\draw (6,4) -- (4,4);
\draw (1.5,0.5) -- (1.5,2.5) -- (3.5,2.5);
\draw (2,3) -- (2,5) -- (1,6);
\draw (3,4) -- (4,3) -- (6,3);
\end{tikzpicture}
\end{center}
\caption{A dyadic tiling of the unit cube in which no half of the cube is
  tiled. Six tiles of dimensions $1\times\tfrac12\times\tfrac14$ (with
  various orientations) are shown, and the remaining space is filled by two
  $\tfrac12\times\tfrac12\times\tfrac12$ cubes.}
\label{F:3d}
\end{figure}
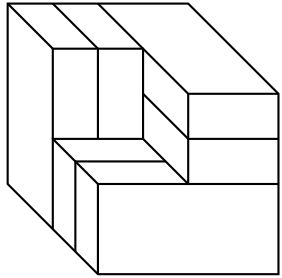

\begin{Cor}\label{triv}
We have  $T_{n+1}(p) \leq 2T_n(p)^2$.
\end{Cor}

\begin{proof}
A tiling of a rectangle such as $[0,1]\times [0,1/2]$ by order $n+1$ tiles is
isomorphic under an obvious affine transformation to a tiling of the unit
square by order-$n$ tiles. Therefore the probability that the two horizontal
rectangles $[0,1]\times [0,1/2]$ and $[0,1]\times [1/2,1]$ can both be tiled
by available order-$(n+1)$ tiles is $T_n(p)^2$, and similarly for the two
vertical rectangles $[0,1/2]\times [0,1]$ and $[1/2,1]\times [0,1]$. The
inequality now follows from \cref{L:JRS}.
\end{proof}

\begin{Cor}\label{C:leftSharpness}
For any given $p$, if there is an $n$ such that $T_n(p)<1/2$, then
$\lim_{n\to\infty}T_n(p) = 0$.
\end{Cor}

\begin{proof}This is immediate from \cref{triv}.\end{proof}

We remark that the threshold $1/2$ in Corollary~\ref{C:leftSharpness} can be
improved (although this will not be needed). Consider the event
$T^\textrm{vert}$ that there is a tiling by available order-$n$ tiles of both
vertical halves of the unit square, and the event $T^\textrm{horiz}$ that
there is a tiling of both horizontal halves. Since $T^\textrm{vert}$ and
$T^\textrm{horiz}$ are increasing events, by the Harris-FKG inequality
\cite{H60},
\[
\P(T^\textrm{horiz}\cap T^\textrm{vert}) \ge
\P(T^\textrm{horiz})\,\P(T^\textrm{vert}) = T_n(p)^4.
\]
Thus we get the inequality $T_{n+1}(p) \leq 2T_n(p)^2 - T_n(p)^4$. It follows
that if $T_n(p) < (\surd5-1)/2$ then $\lim_{n\to\infty}T_n(p) = 0$.

Since $T_0(p)=p$, \cref{C:leftSharpness} immediately implies $p_c\geq 1/2$,
and the enhancement described above gives $p_c\geq (\surd5-1)/2$.

Corollary~\ref{C:leftSharpness} also implies that $T_n(p_c)\ge 1/2$ for all
$n$ (and in fact, $1/2$ may be replaced with $(\surd5-1)/2$). Indeed, if
$T_n(p_c)<1/2$ then, by the continuity of $T_n$ (which is a polynomial), for
some $p>p_c$ we would also have $T_n(p)<1/2$. Corollary~\ref{C:leftSharpness}
then shows that $T_n(p)\to 0$, in contradiction to $p>p_c$.

\subsection*{Covering and lower bounds}

Next we briefly discuss covering.  This basic concept will not be needed for
our proof of Theorem~\ref{T:7/8}, but it motivates parts of the proof and
also yields improved lower bounds on $p_c$.  If some point $x\in [0,1]^2$ is
not covered by any available tile, then clearly there is no tiling by
available tiles. Each of the the $4^n$ squares of size $2^{-n}\times 2^{-n}$
is uncovered with probability $(1-p)^{n+1}$, and it follows that for $p>3/4$
there are no uncovered points with high probability as $n\to\infty$. A
standard second-moment argument shows furthermore that for $p<3/4$ there are
uncovered points with high probability, implying $p_c\geq3/4$.

The absence of uncovered points is necessary but not sufficient
for tiling; see Figure~\ref{F:covtil} for an example. Therefore
the above argument cannot yield an upper bound on $p_c$. In
fact, it may be shown that $p_c$ (the critical point for
tiling) is strictly greater than $3/4$ (the critical point for
covering), as follows. Let $n\geq 1$.  A \df{friend} of an
order-$n$ tile $s$ is an order-$n$ tile $t$ such that $s\cup t$
is an order-$(n-1)$ tile. Observe that in any dyadic tiling,
every tile has some friend also present in the tiling. Call a
$2^{-n}\times 2^{-n}$ square \df{bad} if every available tile
that covers it has no available friend.  Thus a bad square
prevents tiling.  Each square is bad with probability
$[1-p+p(1-p)]^2 [1-p+p(1-p)^2]^{n-1}$, and a second-moment
argument again goes through to show that if $1-p+p(1-p)^2>1/4$
then bad squares exist with high probability as $n\to\infty$.
This gives $p_c\geq 0.785996$.  It seems likely that the bound
could be further improved by considering more complicated local
obstacles to tiling.

\begin{figure}
\begin{center}
\includegraphics[width=4.2cm]{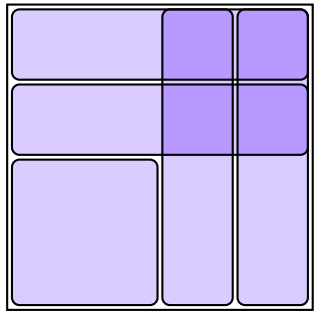}
\end{center}
\caption{A configuration that covers but does not tile.}
\label{F:covtil}
\end{figure}

Before moving on to discuss the proof of \cref{T:7/8} we explain the claimed
lower bound on $p_c(3)$ in dimension $3$.  We cannot use covering: for $d\geq
3$, every point of $[0,1]^d$ is covered with high probability for {\em every}
$p>0$ (by a first-moment argument).  However, in a tiling of the cube
$[0,1]^3$ by order-$n$ tiles, intersecting each tile with a fixed face of the
cube yields a tiling of the face by $2$-dimensional tiles of orders $n$
\textit{or lower}, and furthermore any such $2$-dimensional tile can arise
from exactly one possible $3$-dimensional order-$n$ tile. Let $L_n(p)$ be the
probability that the square $[0,1]^2$ is tiled by available tiles when each
tile of order $n$ or lower is available independently with probability $p$.
Thus the $3$-dimensional tiling probability $T_n^{(3)}(p)$ is at most
$L_n(p)$. Similarly to \cref{triv}, we obtain
$$L_{n+1}(p) \leq p+2L_n(p)^2,$$
and it follows that if $p<1/8$ then $\limsup_{n\to\infty} L_n(p)<1$.  Thus
$p_c(3)\geq 1/8$.  (Again, this bound could likely be improved).

\subsection*{Outline of proof}

We next describe the main ideas behind the proof of Theorem~\ref{T:7/8}. The
basic strategy is simple and standard: we show that if $[0,1]^2$ is not
tiled, then a certain combinatorial structure of unavailable tiles must
exist; by counting such structures (weighted according to the number of
unavailable tiles) we then show that for $p$ sufficiently close to $1$ their
expected number is small.  The challenge, of course, is to find a suitable
class of combinatorial structures.

As discussed above, uncovered points (or similar structures) are not suitable
for our purpose, because their absence does not imply existence of a tiling.
Instead we proceed as follows.  If the unit square is not tileable by
available tiles, then by Lemma \ref{L:JRS}, one of the two horizontal halves
and one of the two vertical halves must also be not tileable; see e.g.\
Figure \ref{F:blocking}(a).  We can then iterate: each of the two
non-tileable halves must itself have two non-tileable halves, and so on until
we reach some ``blocking set'' of unavailable tiles of order $n$, whose
unavailability is sufficient to prevent a tiling of the square. For example,
Figure~\ref{F:blocking}(b) shows one possibility at order $2$.

If at every stage of the above procedure all the resulting non-tileable tiles
were distinct, then the proof would be straightforward: the number of
unavailable tiles in the final blocking set would be $2^n$, and the number of
possible blocking sets would be at most $4^{2^n-1}$ (there are $4$ choices
for the pairs of halves of each tile), and $4^{2^n-1} (1-p)^{2^n}$ is small
for $p$ sufficiently close to $1$.

However, the tiles resulting from the above iterative procedure are {\em not}
necessarily distinct. Even at order $2$, there is a blocking set of $3$ (as
opposed to $4$) tiles; see Figure~\ref{F:blocking}(c). A blocking set with
fewer tiles signals a potential difficulty, since the probability that they
are all unavailable is larger. However, the number of possible outcomes with
fewer tiles may also be smaller. In particular, the {\em minimum} number of
unavailable tiles of order $n$ needed to prevent tiling the unit square is
$n+1$, and in fact the sets of $n+1$ tiles that achieve this are precisely
those whose mutual intersection is some $2^{-n}\times 2^{-n}$ square.
Therefore the number of such minimum-size blocking sets is only $4^n$, and
ruling them out for large $p$ simply amounts to the earlier ``covering''
calculation.

\begin{figure}
\centering

\subfigure[]{
\includegraphics[width=3.2cm]{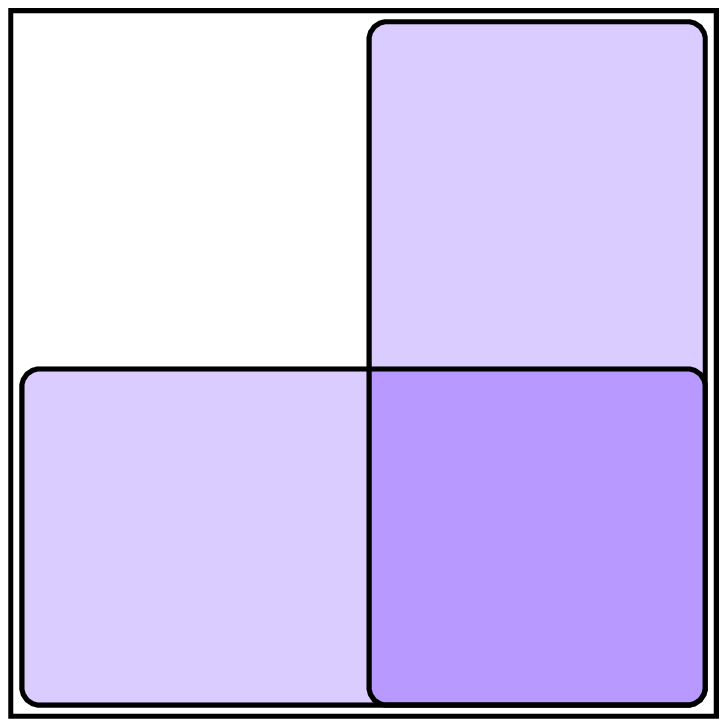}
}
\subfigure[]{
\includegraphics[width=3.2cm]{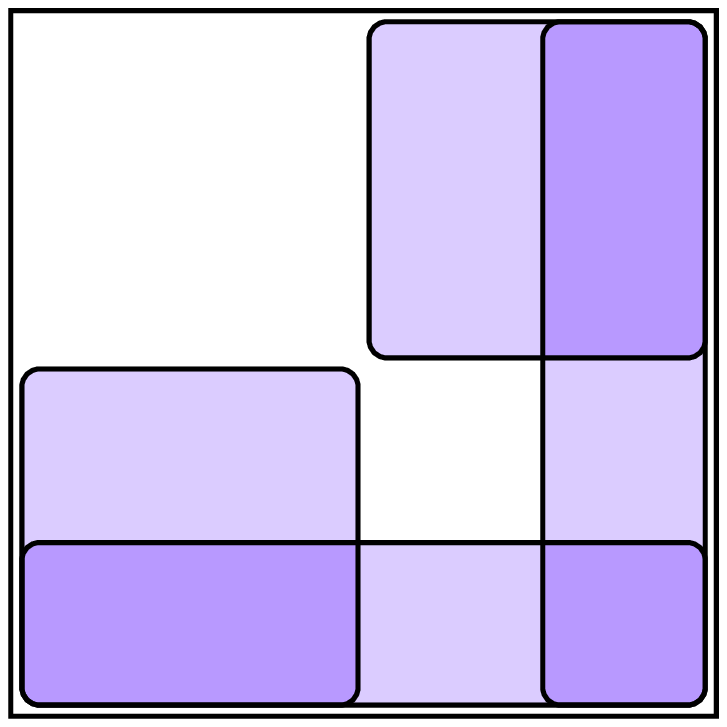}
}
\subfigure[]{
\includegraphics[width=3.2cm]{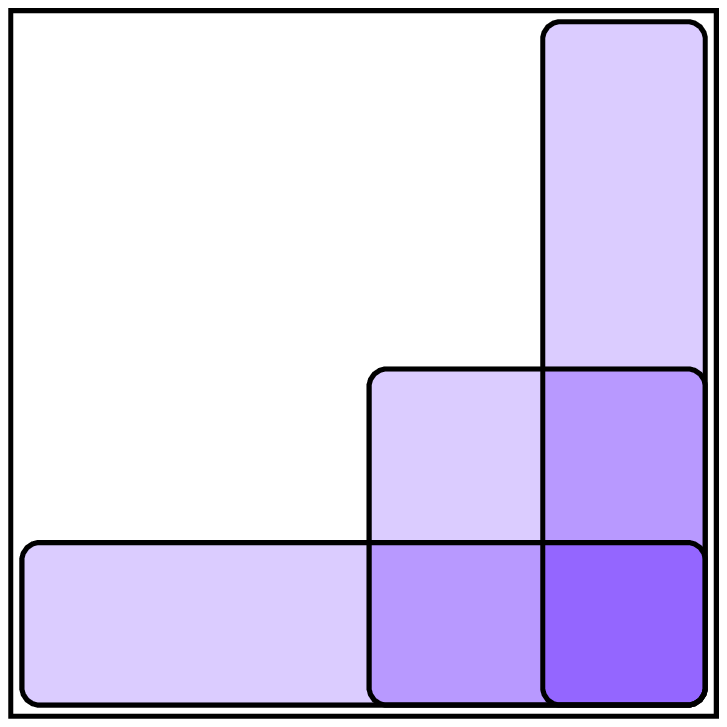}
}

\caption{(a) Two order-$1$ tiles that together block the unit square. (b)
Four order-$2$ tiles that block the tiles in (a). (c) Three order-$2$ tiles
that block the tiles in (a).} \label{F:blocking}
\end{figure}

The issue now is that there are many intermediate blocking sets with numbers
of tiles between $n+1$ and $2^n$.  We must analyze these possibilities taking
into account both the number of choices and the resulting numbers of tiles
(and these two quantities must be weighed against each other).  To achieve
this, we will organize blocking sets into {\em chains}.  A chain is the set
of all tiles of a given order that contain some fixed tile of some higher
order.  (Equivalently, it is one of the minimum blocking sets discussed
above, but within some tile of intermediate order rather than the whole
square). For example, the set in Figure \ref{F:blocking}(c) is a chain of 3
tiles, while that in (b) can be expressed as a union of 2 chains each
consisting of 2 tiles. Although our eventual interest lies in the cardinality
of the blocking set resulting from the iterative procedure, we will count the
possible outcomes using a generating function of {\em two} parameters,
corresponding to numbers of tiles and numbers of chains.  The resulting
counting argument is short but somewhat mysterious.  The inclusion of both
parameters appears to be not merely a technical requirement but a fundamental
one: we do not know how to proceed by counting tiles alone.

Another complication is as follows.  In the iterative procedure for finding
blocking sets outlined above, several choices may be possible.  It is
possible for example that {\em both} horizontal halves of a tile are
non-tileable, and we must choose one of them.  It turns out that how we do
this is crucial.  We will impose the rule that we always make the choice that
minimizes the resulting number of chains.  For example, starting from the
situation of Figure \ref{F:blocking}(a), we choose the blocking set in (c) in
preference to the one in (b).  (We might however be forced to take the one in
(b) if the bottom-right $\tfrac12\times\tfrac12$ square is tileable). With
this rule, it will turn out that the collection of chains produced by the
iterative procedure is pairwise disjoint.  Without it, chains could intersect
(or indeed coincide); this would result in a reduction in the number of tiles
in the blocking set, again adversely affecting the resulting bound on the
probability that they are all unavailable.  Combined with the counting
argument mentioned earlier, this disjointness of chains suffices to give a
bound of the required form.

\subsection*{Organization}  The article is organized as
follows.  In \cref{sec:sharp} below we prove the sharp threshold result,
\cref{T:sharp}.  The remainder of the article is devoted to the proof of
\cref{T:7/8}.  In \cref{S:TilesAndChains} we introduce {\em chains} of tiles,
and prove some properties.  In \cref{S:chaintrees} we arrange chains into
{\em chain trees}.   These are the ``blocking configurations'' at the heart
of the proof: if the unit square is not tiled, then there is a blocked chain
tree, and we can bound the probability of this event by counting chain trees.
In \cref{S:principal} we prove that non-tileability implies the existence of
a chain tree of a special type which is guaranteed to have all its chains
disjoint, as discussed above.  Finally, in \cref{analysis} we employ
generating functions to perform the necessary counting argument.   We
conclude with some open problems.

\section{Sharp threshold}\label{sec:sharp}

We will deduce \cref{T:sharp} from \cref{C:leftSharpness} together with the
following result of Friedgut and Kalai. In their paper \cite{FK96} this is
Theorem~2.1, modified according to the comment after Corollary~3.5. Here $A$
is a subset of the hypercube $\{0,1\}^N$, endowed with the product
probability measures $\P_p$.

\begin{Thm*}[Friedgut and Kalai; \cite{FK96}]\label{T:FK}
Let $A$ be increasing and invariant under the action on $\{1,\ldots,N\}$ of a
group with orbits of size at least $m$. If $\P_p(A) > \eps$ then $\P_q(A) >
1-\eps$ for $q = p + c \log(1/2\eps) / \log m$ and an absolute constant $c$.
\end{Thm*}

\sloppypar (For our application, all that matters is that the difference
$q-p$ appearing in the above theorem tends to $0$ as $m\to\infty$.  As noted
in \cite{FK96}, this can also be deduced from earlier results of Russo
\cite{russo} or Talagrand \cite{T94}.)

\begin{proof}[Proof of \cref{T:sharp}]
The second claim of \cref{T:sharp} is immediate since $T_n(p)$ is increasing
in $p$, so we turn to the first claim. Let $S$ be the set of order-$n$ tiles,
and let $A\subset\{0,1\}^S$ be the event that the unit square is tileable by
available tiles, where $0$ and $1$ represent unavailable and available
respectively. Thus $T_n(p) = \P_p(A)$. We will show below that $A$ is
invariant under the action of a group of permutations of tiles with orbits of
size $2^n$.

Suppose that $p<p_c$, and let $p<q<p_c$. The definition of $p_c$ implies that
$T_n(q) < 1-\eps$ for some $0<\eps<1/2$ and infinitely many $n$. Since
$c\log(1/2\eps)/\log(2^n) \to 0$ as $n\to\infty$, the Friedgut-Kalai theorem
then implies that $T_n(p) \leq \eps < 1/2$ for some $n$.
Corollary~\ref{C:leftSharpness} then gives $\lim_{n\to\infty} T_n(p)=0$ as
required.

It remains to exhibit a group of symmetries of $A$ with orbits of size $2^n$.
Consider the mapping $f_k:[0,1]\to[0,1]$ that changes the $k^\textrm{th}$
digit in the binary expansion, leaving the rest unchanged. It is easy to see
that for any $k$, the maps $(x,y) \mapsto (f_k(x),y)$ and $(x,y) \mapsto
(x,f_k(y))$ both permute dyadic tiles, since specifying a dyadic tile is
equivalent to specifying several initial digits in each of $x$ and $y$ (see
\cref{F:h1}).

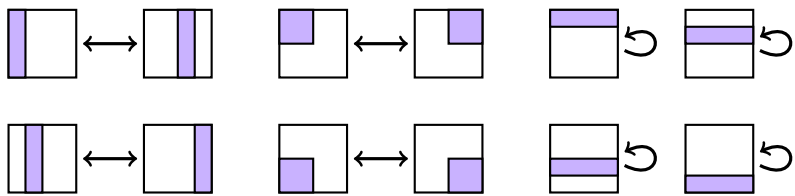
\begin{figure}
\begin{center}
\begin{tikzpicture} [thick, fill=blue!70!magenta!30,scale=0.9]
\begin{scope} [yshift=1.7cm]
\draw (0,0) rectangle +(1,1);
\filldraw (0,0) rectangle +(.25,1);
\draw [<->, very thick] (1.1,.5) -- (1.9,.5);
\draw (2,0) rectangle +(1,1);
\filldraw (2.5,0) rectangle +(.25,1);
\end{scope}
\begin{scope}
\draw (0,0) rectangle +(1,1);
\filldraw (.25,0) rectangle +(.25,1);
\draw [<->, very thick] (1.1,.5) -- (1.9,.5);
\draw (2,0) rectangle +(1,1);
\filldraw (2.75,0) rectangle +(.25,1);
\end{scope}
\begin{scope} [yshift=1.7cm, xshift=4cm]
\draw (0,0) rectangle +(1,1);
\filldraw (0,.5) rectangle +(.5,.5);
\draw [<->, very thick] (1.1,.5) -- (1.9,.5);
\draw (2,0) rectangle +(1,1);
\filldraw (2.5,.5) rectangle +(.5,.5);
\end{scope}
\begin{scope} [xshift=4cm]
\draw (0,0) rectangle +(1,1);
\filldraw (0,0) rectangle +(.5,.5);
\draw [<->, very thick] (1.1,.5) -- (1.9,.5);
\draw (2,0) rectangle +(1,1);
\filldraw (2.5,0) rectangle +(.5,.5);
\end{scope}
\begin{scope} [yshift=1.7cm, xshift=8cm]
\draw (0,0) rectangle +(1,1);
\filldraw (0,.75) rectangle +(1,.25);
\draw [->, very thick] (1.1,.4) .. controls (1.7,.1) and (1.7,.9) .. (1.1,.6);
\end{scope}
\begin{scope} [xshift=8cm]
\draw (0,0) rectangle +(1,1);
\filldraw  (0,.25) rectangle +(1,.25);
\draw [->, very thick] (1.1,.4) .. controls (1.7,.1) and (1.7,.9) .. (1.1,.6);
\end{scope}
\begin{scope} [yshift=1.7cm, xshift=10cm]
\draw (0,0) rectangle +(1,1);
\filldraw  (0,.5) rectangle +(1,.25);
\draw [->, very thick] (1.1,.4) .. controls (1.7,.1) and (1.7,.9) .. (1.1,.6);
\end{scope}
\begin{scope} [xshift=10cm]
\draw (0,0) rectangle +(1,1);
\filldraw (0,0) rectangle +(1,.25);
\draw [->, very thick] (1.1,.4) .. controls (1.7,.1) and (1.7,.9) .. (1.1,.6);
\end{scope}
\end{tikzpicture}
\end{center}
\caption{The symmetry $(x,y)\mapsto (f_1(x),y)$ swaps the left and
right halves of the square, and acts on the $12$ tiles of order $2$.}
\label{F:h1}
\end{figure}

Since these maps preserve intersection of tiles, they also preserve
tilings, and hence they preserve $A$. By applying a sequence of such maps
we can change any tile of a given shape to any other tile of the same
shape, and so the generated group has orbits of size $2^n$.
\end{proof}

In dimensions $d\geq 3$ we do not know whether the analogue of \cref{T:sharp}
holds, since our proof relies on \cref{L:JRS}. However, we can still apply
the Friedgut-Kalai theorem as above for all $d$.  Hence the argument for the
bound $p_c(3)\geq 1/8$ also gives that $T_n^{(3)}(p)\to 0$ for $p<1/8$.
\enlargethispage*{1cm}

\section{Blocked tiles and chains} \label{S:TilesAndChains}

Our next objective is to prove Theorem~\ref{T:7/8}.  We start with some
important definitions.

Given a classification of the order-$n$ tiles into available and
unavailable, we say that a tile of order $k\leq n$ is \df{tileable} if
it can be tiled by available tiles of order $n$. Otherwise it is \df{blocked}
(in particular, tiles of order $n$ are blocked if and only if
they are unavailable).
%
The two order-$(k+1)$ tiles that are obtained by bisecting an
order-$k$ tile with a horizontal cut are its \df{horizontal children}, and
similarly the two tiles that are obtained by cutting vertically are its
\df{vertical children}.

\begin{Lemma}\label{L:trivial}
A tile of order less than $n$ is blocked if and only if at least one of its
horizontal children and at least one of its vertical children is blocked.
\end{Lemma}

\begin{proof}
This is the contrapositive of Lemma~\ref{L:JRS} (applied to a tile rather
than the whole square).
\end{proof}

Our goal is to arrange sets of blocked order-$k$ tiles into {\em chains}.
If $s$ and $t$ are tiles of order $k$ whose interiors intersect, then the
\df{chain} $[s,t]$ is the set of all order-$k$ tiles that contain their
intersection $s\cap t$.

A chain contains a most horizontal and a most vertical tile (these are $s$
and $t$), and exactly one tile of each intermediate shape. Two order-$k$
tiles are called \df{adjacent} if their intersection is an order-$(k+1)$
tile. A \df{bond} of a chain is a pair of adjacent tiles in the chain.  Thus
the number of bonds of a chain is one less than the number of tiles.  Observe
also that a chain is precisely a directed path in the graph whose vertices
are all order-$k$ tiles, with adjacent pairs connected by an edge directed
towards the more vertical tile.  (In fact, this graph may be viewed as a
lamplighter graph corresponding to binary lamps on a path of length $n$; see
e.g.\ \cite{woess} for a definition.)

\begin{figure}
\begin{center}
\includegraphics[width=3.5cm]{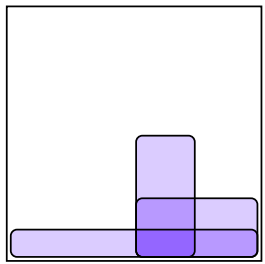}
\end{center}
\caption{A chain consisting of the three order-$3$ tiles that contain an order-$5$ tile.}
\label{F:chain}
\end{figure}

We now define the notion of successors of chains. Let $[s,t]$ be an order-$k$
chain, with $s$ the horizontal end-tile and $t$ the vertical end-tile. A
\df{successor} of the chain $[s,t]$ is any set of tiles of order $k+1$ with
the property that it includes exactly one horizontal child and exactly one
vertical child of every element of $[s,t]$.

The idea of the last definition is that a successor of $[s,t]$
is a minimal set with the property that if it were blocked,
then $[s,t]$ would be blocked, according to \cref{L:trivial}.
(We call a set of tiles blocked if all its tiles are blocked).
If a chain of order less than $n$ is blocked, then it possesses
some blocked successor. The last fact is not immediately
obvious, because of the requirement that a successor include
{\em exactly} one child of each type.  However, it follows from
\cref{F:bridge} and the proof of \cref{L:successor} below.  We
will eventually need a somewhat stronger statement---see
\cref{L:principal} in \cref{S:principal}.

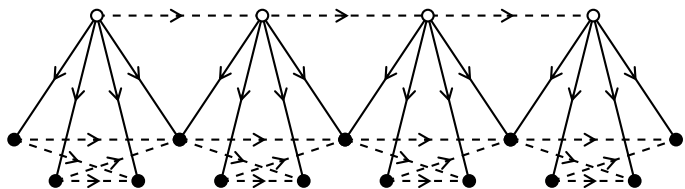
\begin{figure}
\begin{center}
\begin{tikzpicture} [scale=0.55]
\begin{scope}[thick, every node/.style={sloped,allow upside down}]
\draw[dashed] (2,4) -- node {\midarrow} (6,4)
-- node {\midarrow} (10,4) -- node {\midarrow} (14,4);
\foreach \x in {0,4,8,12}
{
\begin{scope}[shift={(\x,0)}]
\draw (2,4) -- node {\midarrow} (0,1);
\draw (2,4) -- node {\midarrow} (1,0);
\draw (2,4) -- node {\midarrow} (3,0);
\draw (2,4) -- node {\midarrow} (4,1);
\draw[dashed] (0,1) -- node {\midarrow} (4,1);
\draw[dashed] (1,0) -- node {\midarrow} (4,1);
\draw[dashed] (0,1) -- node {\midarrow} (3,0);
\draw[dashed] (1,0) -- node {\midarrow} (3,0);
\filldraw[fill=white] (2,4) circle (4pt);
\filldraw (0,1) circle (4pt);
\filldraw (1,0) circle (4pt);
\filldraw (3,0) circle (4pt);
\filldraw (4,1) circle (4pt);
\end{scope}
}
\end{scope}
\end{tikzpicture}
\end{center}
\caption{Tiles involved in possible successors of a chain.
White discs are the $4$ tiles of a chain of order $k$ (with $3$ bonds).
Black discs are tiles of order $k+1$.  Adjacent tiles
are connected by dashed arrows towards the more vertical tile.
Solid arrows connect a tile to its children
(with horizontal children to the left, and vertical children to the right).}
\label{F:bridge}
\end{figure}

A key ingredient in our proof is to classify the possible successors of a
given chain. We say that two chains are \df{separate} if they are disjoint,
and no tile of one is adjacent to any tile of the other.  (This implies in
particular that their tiles cannot be partitioned into one or two chains in
any other way).  Here and elsewhere, disjointness of chains of means simply
that they have no tiles in common; the tiles themselves are permitted to
intersect one another.

\begin{Lemma}\label{L:successor}
Any successor of a chain can itself be uniquely expressed as a union of
pairwise separate chains.  For a chain of $b$ bonds, any successor has
exactly $b+1$ bonds in total, and there are $4 \binom{b}{r}$ possible
successors that consist of $r+1$ separate chains, for each $r=0,\dots,b$.
\end{Lemma}

\begin{proof}
The key observations are illustrated in Figure~\ref{F:bridge}.  Each tile in
a chain $[s,t]$ has two horizontal children and two vertical children, but
not all these children are distinct. Specifically, if $u,v$ are two adjacent
tiles of the chain $[s,t]$ (with $u$ the more horizontal), then there is a
unique tile that is both a vertical child of $u$ and a horizontal child of
$v$, namely the intersection $u\cap v$.  Aside from such intersections (one
for each bond of $[s,t]$), all children of the tiles of $[s,t]$ are distinct.
Note also that any horizontal child of a given tile is adjacent to any
vertical child, and these are the only adjacencies among children of the
tiles of $[s,t]$.

We can now consider possible successors.  Firstly, if for each
bond in $[s,t]$ we take the intersection tile, and in addition
we choose one horizonal child $s'$ of $s$ and one vertical
child $t'$ of $t$, then we obtain one possible successor of
$[s,t]$---in fact, this successor is precisely the
order-$(k+1)$ chain $[s',t']$. We call a successor consisting
of a single chain \df{simple}. See \cref{F:successors}(b) for
an example. There are $2^2=4$ possible simple successors of a
given chain, since there are two possibilities each for $s'$
and $t'$.

On the other hand, consider a successor that does not include $u\cap v$. (In
the forthcoming application to blocked chains, it will be necessary to
consider such a case if $u\cap v$ is tileable). In that case the successor
must include the other vertical child of $u$, namely $\overline{u\setminus
v}$ (where the bar denotes topological closure), and similarly it must
include $\overline{v\setminus u}$.  Now if, for instance, for each of the
other bonds of $[s,t]$ we select the intersection tile as before (and we
select the same end tiles $s',t'$), the resulting successor can be expressed
as the union of the two separate chains $[s',\overline{u\setminus v}]$ and
$[\overline{v\setminus u},t']$. We say that a \df{split} occurred at the bond
$(u,v)$. In general, each bond of $[s,t]$ may or may not be split, and the
resulting successor can always be uniquely expressed as a union of separate
chains, with $r$ splits resulting in $r+1$ chains. Combined with the $4$
choices of the end tiles $s'$ and $t'$, this gives the claimed enumeration.
\end{proof}

\begin{figure}
\centering

\subfigure[]{
\includegraphics[width=2.8cm]{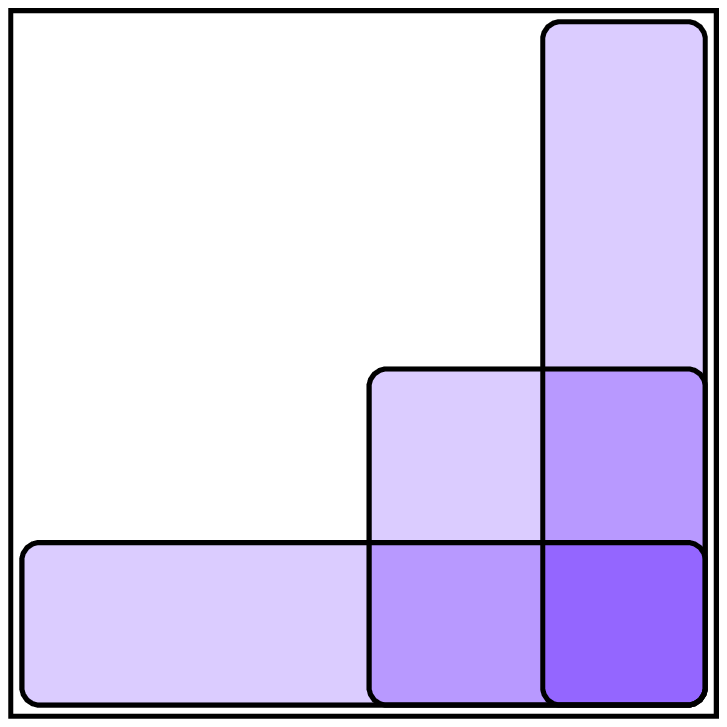}
} \hfill
\subfigure[]{
\includegraphics[width=2.8cm]{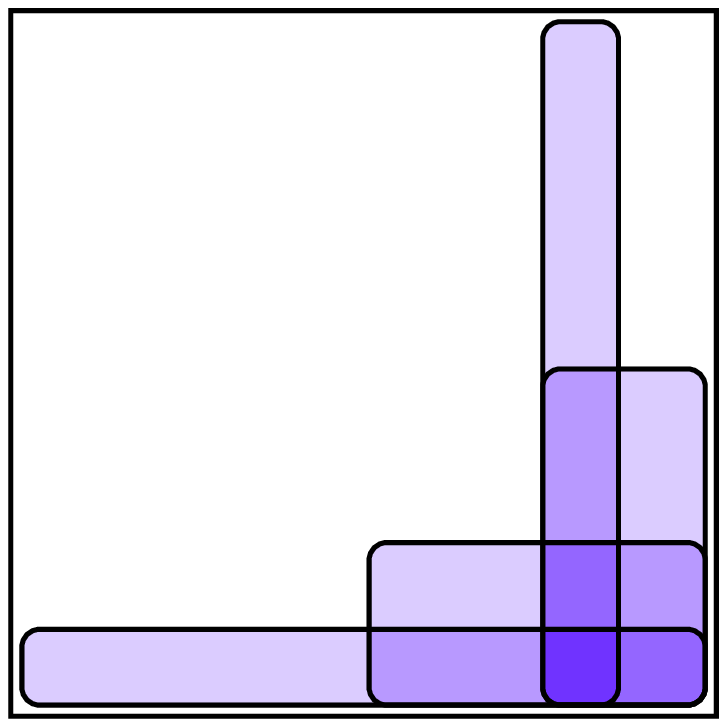}
} \hfill \subfigure[]{
\includegraphics[width=2.8cm]{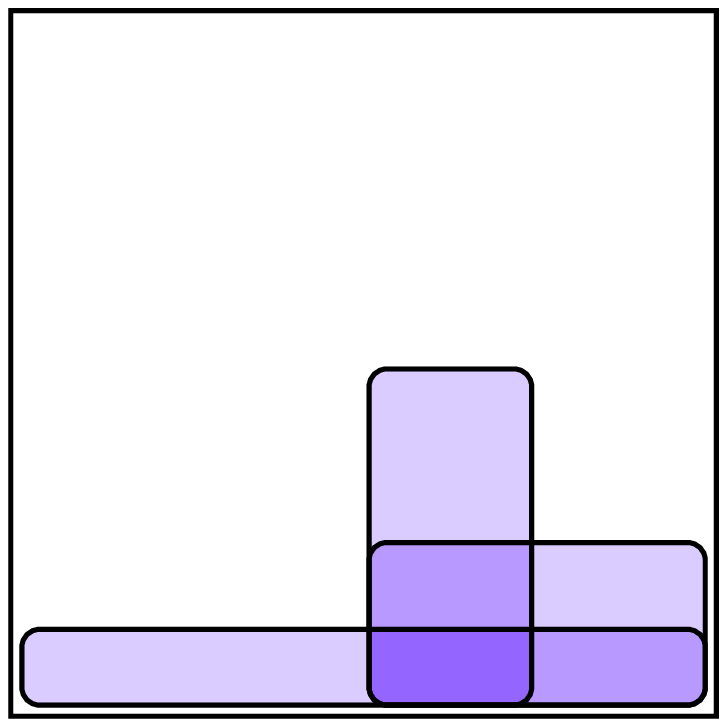}
\includegraphics[width=2.8cm]{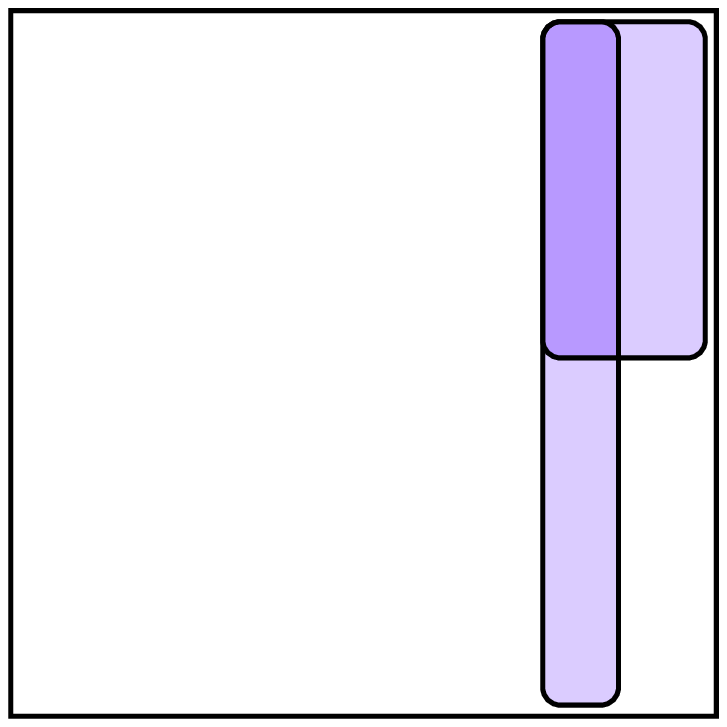}
}

\caption{(a) A chain of tiles of order $2$. (b) A simple successor of the
chain. (c) A split successor of the chain, with the same end tiles.}
\label{F:successors}
\end{figure}

See Figure \ref{F:successors} for an example---if one splits
the order-2 chain $[s,t]=\{s,u,t\}$ in (a) between tiles $u$
and $t$ (i.e.~between the square and the vertical tile), one
gets two chains shown in (c), the first being a simple
successor of $[s,u]$ and the second a simple successor of
$[t,t]$.

The above ideas will be applied as follows. If every tile of a chain of order
less than $n$ is blocked, then the chain must have a successor each of whose
tiles is blocked, by Lemma~\ref{L:trivial}.  Thus, if the unit square is
blocked, then we can start from the chain consisting only of the unit square,
and repeatedly find successors until we reach a set of chains of order $n$
consisting entirely of unavailable tiles.  The set of tiles in these chains
has the property that if they are unavailable then the square is not
tileable.  Next we want to count the possible outcomes of such a process.

\section{Chain trees}\label{S:chaintrees}

We next introduce an object called a chain tree, which formalizes the idea of
an iterative construction of a blocking set of tiles.  Note however that the
definition itself will be purely combinatorial, and will not refer to
availability of tiles.

A \df{chain tree} of depth $n$ is a rooted tree of depth $n$ in which each
vertex is labeled with a chain of tiles (where we allow the possibility that
distinct vertices are labeled with the same chain or intersecting chains),
and with the following properties.

\begin{enumerate}
\item[(I)] The root corresponds to the order-$0$ chain consisting only of
the unit square.
\item[(II)] For any vertex $v$ at level less than $n$, the children of
    $v$ correspond to pairwise separate chains whose union is a successor
    of the chain at $v$.
\end{enumerate}

Note that each vertex at level $k$ corresponds to a chain of order $k$, and
the leaves of the tree correspond to chains of order $n$.  Observe also that
stripping the leaves from a chain tree of depth $n+1$ results in a chain tree
of depth $n$.

For a chain tree $T$ of depth $n$, let $c(T)$ be the number of leaves, and
$t(T)$ the total number of tiles in chains at the leaves (counted with
multiplicity; in other words the sum of the cardinalities of the chains
rather than the number of distinct tiles that occur). It is also convenient
to let $b(T)=t(T)-c(T)$ be the total number of bonds in the chains at the
leaves (also counted with multiplicity).  For example, for the simplest chain
trees consisting only of simple successors, and exactly one chain at each
level, we have $c(T)=1$ and $b(T)=n$.

We will need to enumerate chain trees of depth $n$ weighted according to the
number of leaf tiles. To this end we define the two-variable polynomial
\[
f_n = f_n(q,z) := \sum_T q^{b(T)} z^{c(T)},
\]
where the sum is over all chain trees of depth $n$.
For instance, $f_0 = z$, since the unique depth zero chain tree
consists of a single chain containing a single tile (so no
bonds). At order $1$ there are $4$ possible successors to this
chain, and there cannot yet be any splitting. Hence $f_1 = 4 q
z$. At the next level, any given order-$1$ chain has $4$
possible simple successors, and $4$ possible split successors
into $2$ chains (each having one bond). Hence $f_2 = 4(4q^2z +
4q^2z^2) = 16 q^2 z (1+z)$.

Next we will see how $f_n$ is related to the tiling probability
$T_n$.

\section{The principal chain tree}
\label{S:principal}

In this section we will prove the following.
\begin{Prop}\label{P:Tbound}  For all $n,p$ we have
\begin{equation}\label{eq:Tbound}
1 - T_n(p) \leq f_n(1-p,1-p).
\end{equation}
\end{Prop}
Then in \cref{analysis} we will analyze the asymptotic behaviour of
$f_n(q,q)$ for $q$ small, which will enable us to deduce \cref{T:7/8}.

To motivate the proof of \cref{P:Tbound}, observe that
$$f_n(q,q) = \sum_T q^{t(T)},$$ which enumerates chain trees
weighted by $q$ to the number of tiles at depth $n$.  If we set
$q=1-p$, and if it happens that the chains at depth $n$ of $T$
are pairwise disjoint, then the term $q^{t(T)}$ is the
probability that all their constituent tiles are unavailable.

If it were the case that the leaves of a chain tree always
corresponded to pairwise disjoint chains, then \cref{P:Tbound}
would follow immediately from \cref{L:trivial} by the argument
outlined at the end of \cref{S:TilesAndChains}.  However, there
do exist chain trees with repeated tiles (see \cref{F:repeated}
for an example).
\begin{figure}
\begin{center}
\begin{tikzpicture}
\node (A) at (0,0) {
\includegraphics[width=2.5cm]{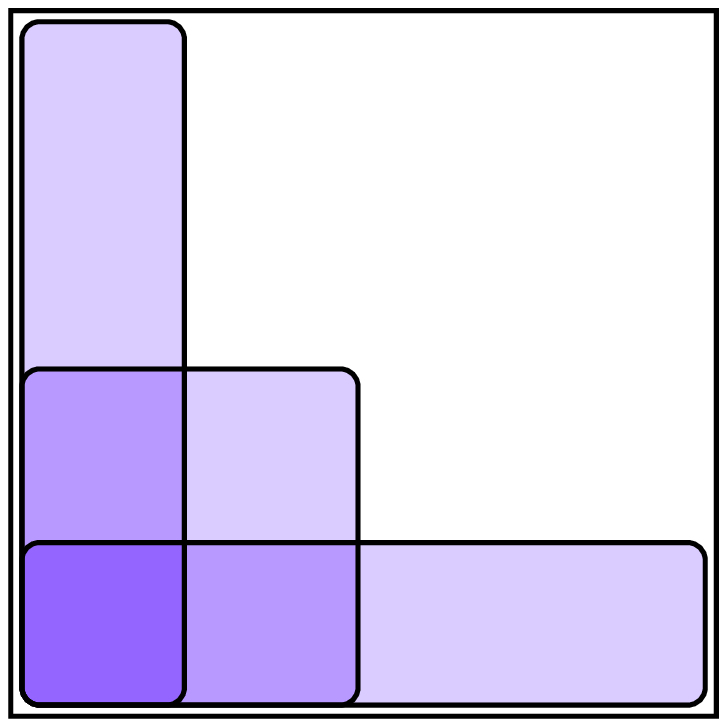}
};
\node (B1) [right=of A, yshift=2.65cm] {
\includegraphics[width=2.5cm]{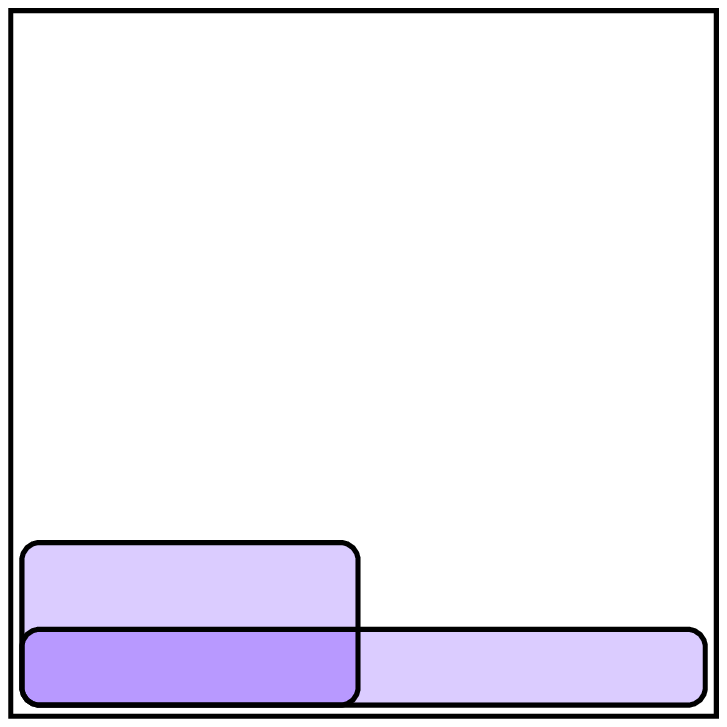}
};
\node (B2) [right=of A] {
\includegraphics[width=2.5cm]{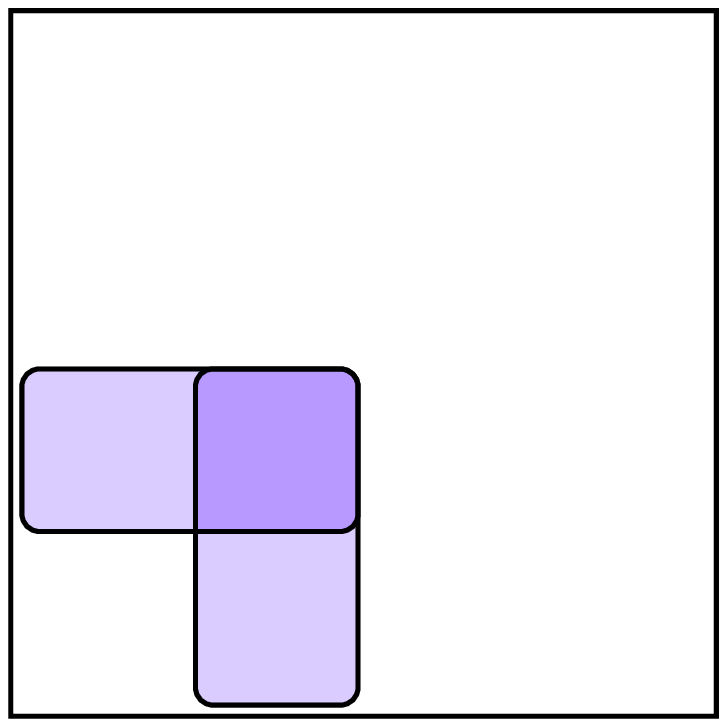}
};
\node (B3) [right=of A, yshift=-2.65cm] {
\includegraphics[width=2.5cm]{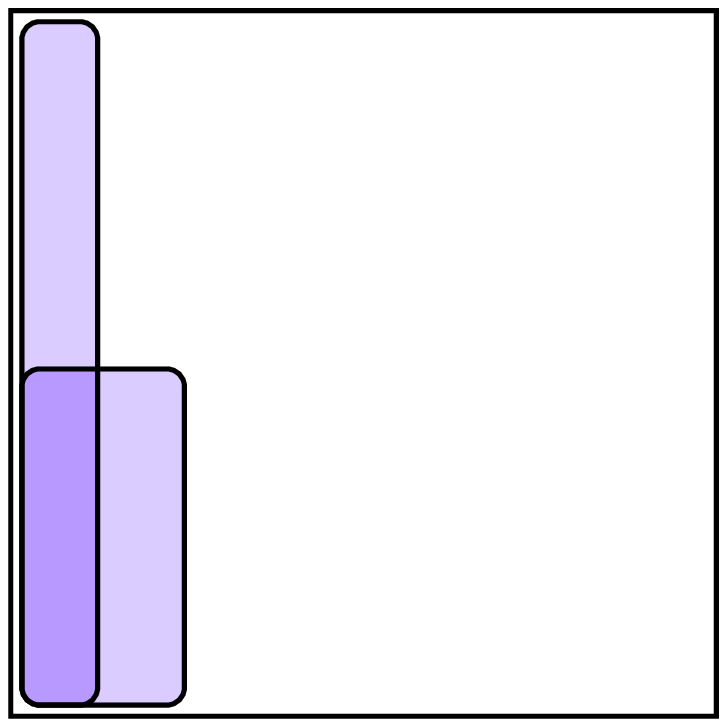}
};
\node (C1) [right=of B1] {
\includegraphics[width=2.5cm]{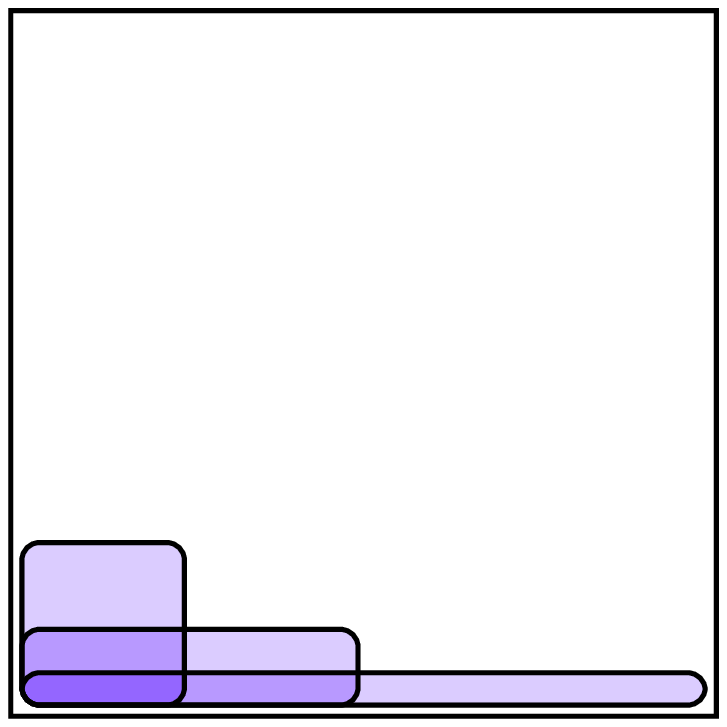}
};
\node (C2) [right=of B2] {
\includegraphics[width=2.5cm]{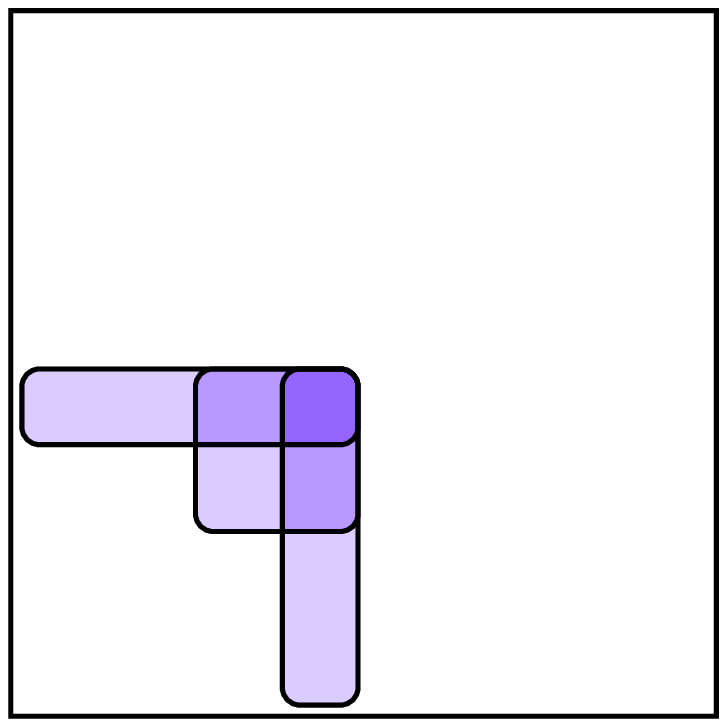}
};
\node (C3) [right=of B3] {
\includegraphics[width=2.5cm]{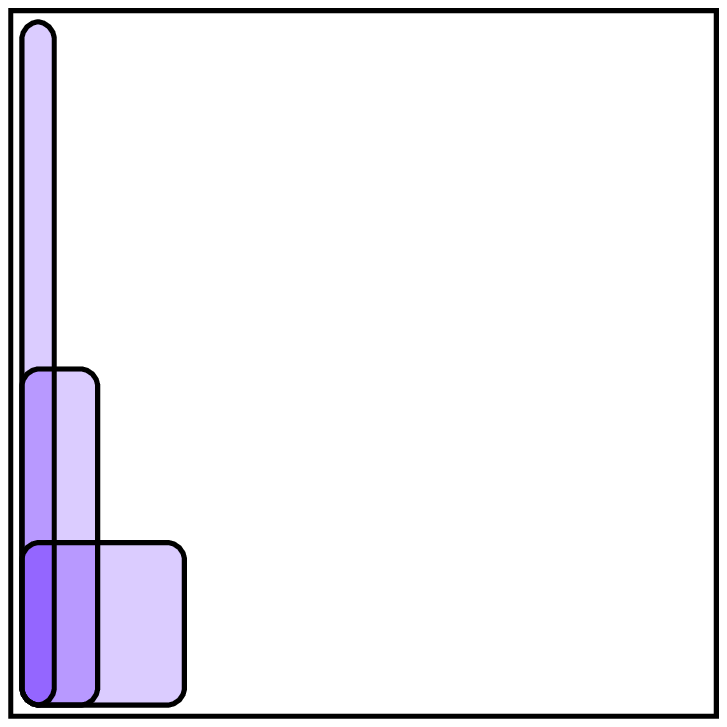}
};
\draw [->, very thick] (A) -- (B1);
\draw [->, very thick] (A) -- (B2);
\draw [->, very thick] (A) -- (B3);
\draw [->, very thick] (B1) -- (C1);
\draw [->, very thick] (B2) -- (C2);
\draw [->, very thick] (B3) -- (C3);
\end{tikzpicture}
\end{center}
\caption{Part of a chain tree with a duplicate tile (arrows are from a node to
its children in the tree). The $\frac14\times\frac14$ bottom left
corner tile is contained in both the top and bottom leaves.}
\label{F:repeated}
\end{figure}

Therefore, we will define a special class of chain trees whose
chains will turn out to be disjoint.  They will be constructed
by iteratively finding blocked successors to each blocked
chain, as discussed earlier, but with the additional
restriction that we \textit{split only where necessary}.  More
formally, given a designation of all order-$n$ tiles as
available and unavailable, we call a depth-$n$ chain tree $T$ a
\df{principal chain tree} if in addition to conditions (I) and
(II), it satisfies the following.
\begin{enumerate}
\item[(III)] Each tile in each chain of $T$ is blocked.
\item[(IV)] If the chain corresponding to some vertex contains a bond
    $(u,v)$ for which the tile $u\cap v$ is blocked, then one of its
    children contains $u\cap v$ in its chain (i.e.\ there is no split at
    this bond).
\end{enumerate}

\begin{Lemma}\label{L:principal}
If $[0,1]^2$ is blocked, then there exists a principal chain tree.
\end{Lemma}

\begin{proof}
Start with the blocked chain containing only $[0,1]^2$, and iteratively find
a blocked successor of each previously-constructed chain, splitting at a bond
$(u,v)$ only when the tile $u\cap v$ is tileable.
\end{proof}

The key fact about principal chain trees is the following.

\begin{Lemma}\label{L:disjoint}
In a principal chain tree, the chains corresponding to distinct vertices are
disjoint.
\end{Lemma}

\begin{proof}
The proof relies on two observations.  First, consider the graph whose
vertices are all tiles (of all orders) that contain as a subset some fixed
tile $w$, and with a directed edge from a tile to its children in this set.
This graph is isomorphic to a rectangular portion of the oriented square
lattice $\mathbb{Z}^2$, as shown in Figure~\ref{F:rectGrid}.  If $w$ has
shape $2^{-a}\times 2^{-b}$, the point $(i,j)$ corresponds to the unique tile
of shape $2^{-i}\times 2^{-j}$ that contains $w$, for each $0\leq i\leq a$
and $0\leq j\leq b$.  (In the figure, the first coordinate increases from top
to bottom, and the second from left to right).  Adjacent pairs of tiles
correspond to points differing by the diagonal vector $(-1,1)$. Thus, a chain
all of whose tiles contain $w$ corresponds to an interval on some diagonal in
the lattice.

\begin{figure}
\begin{center}
\begin{tikzpicture} [scale=1, every node/.style={sloped,allow upside down}]
\draw (6,0) node[below right] {$w$};
\draw (0,4) node[above left] {$[0,1]^2$};
\foreach \x in {0,1,...,6}
\foreach \y in {0,1,...,4}
{\fill [black, xshift=\x cm, yshift=\y cm] (0,0) circle (2pt);}
\foreach \x in {0,1,...,6}
\foreach \y in {1,2,..., 4}
{\draw [xshift= \x cm, yshift=\y cm] (0,0) -- node {\midarrow} (0,-1);}
\foreach \x in {0,1,...,5}
\foreach \y in {0,1,..., 4}
{\draw [xshift= \x cm, yshift=\y cm] (0,0) --node {\midarrow} (1,0);}
\foreach \x in {0,1,...,5}
\foreach \y in {0,1,..., 4}
{\draw [xshift= \x cm, yshift=\y cm] (0,0) --node {\midarrow} (1,0);}
\foreach \x in {0,1,...,5}
\foreach \y in {0,1,..., 3}
{\draw [dashed, xshift= \x cm, yshift=\y cm] (0,0) --node {\midarrow} (1,1);}
\draw [thick, rounded corners, rotate=45, shift={(1.1,-.2)}] (0,0) rectangle
+(4.8,0.4);
\end{tikzpicture}
\end{center}

\caption{The tiles that contain $w$.  Solid arrows point from tiles to their
children, and dashed arrows connect adjacent pairs of tiles, pointing towards
the more vertical tile.  A chain is highlighted.} \label{F:rectGrid}
\end{figure}
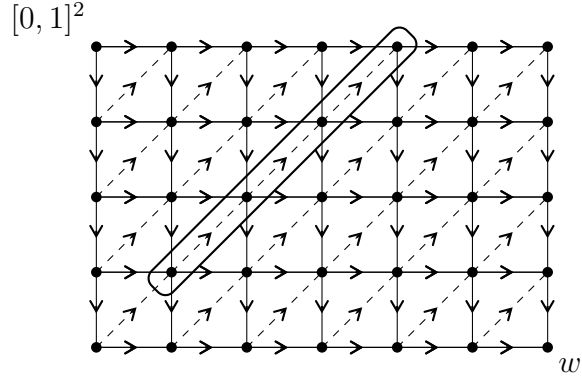

Second, we observe that paths in a chain tree may be mapped to paths in the
lattice in the following way.  Recall that different vertices in the chain
tree may {\em a priori} be labeled with the same chain (although the present
proof will in particular rule this out), so we must be careful to distinguish
between vertices in the tree and their associated chains. Suppose that a tile
$w$ is contained in a chain corresponding to some vertex $x$ in a chain tree.
Consider the unique self-avoiding path $\pi$ from $x$ to the root in the
chain tree.  The chain of the parent vertex of $x$ in this path must contain
either the horizontal parent or the vertical parent of $w$ (possibly both).
(The \df{horizontal parent} of $w$ is the unique tile that has $w$ as a
vertical child, and the \df{vertical parent} is defined similarly). Iterating
this, we find a sequence of blocked tiles, each a parent of the previous one,
starting at $w$ and ending at $[0,1]^2$.  We call such a sequence an
\df{ancestry} of $w$. Each tile of an ancestry belongs to the chain of the
corresponding vertex of the path $\pi$. Each tile in an ancestry contains
$w$, and an ancestry corresponds to a (backwards) directed path in the
lattice.

Now suppose for a contradiction that an order-$n$ tile $w$ occurs in two (not
necessarily different) chains corresponding to different vertices at level
$n$ of a principal chain tree.  These vertices have a last common ancestor
vertex in the chain tree, with a corresponding chain $C$.  We can also find
two ancestries of $w$ corresponding to its membership in the two initial
chains, and each of them must include a tile in $C$; call these two tiles $s$
and $t$, and let $s_+,t_+$ be their respective children in the two
ancestries. By the choice of $C$, the tiles $s_+$ and $t_+$ must lie in the
chains of two different children of $C$.  Hence the chain tree must include a
split at some bond of $C$ somewhere between $s$ and $t$.  By property (IV) of
a principal chain tree, this split must occur because some tile $u$ that is
the intersection of two adjacent tiles in the chain $[s,t]\subseteq C$ is
tileable. Note that $s$ and $t$ both contain $w$, hence so does every tile in
$[s,t]$, and hence so does $u$.

\begin{figure}
\begin{center}
\begin{tikzpicture} [scale=0.8]
\foreach \x in {1,...,10}
\foreach \y in {0,1,...,7}
\fill (\x,\y) circle (1.5pt);
\draw (10,0) node[below right] {$w$};
\draw (1,7) node[above left] {$[0,1]^2$};
\draw (2,3) node[above left] {$s$};
\draw (6,7) node[above left] {$t$};
\draw (5,5) node[below left] {$u$};
\draw (2,3) -> (2,2) -> (3,2) -> (4,2) -> (5,2) -> (5,1) -> (6,1) ->
(7,1) -> (8,1) -> (8,0) -> (9,0) -> (10,0);
\draw[style=double, double distance=2pt] (7,1) -> (8,1);
\draw (6,7) -> (6,6) -> (7,6) -> (7,5) -> (7,4) -> (7,3) -> (7,2) -> (7,1);
\draw (5,5) -> (6,5) -> (6,4) -> (6,3) -> (7,3) -> (8,3) -> (9,3) -> (10,3);
\draw (8,1) -> (9,1) -> (10,1) -> (10,0);
\foreach \x in {(5,5), (6,5), (6,4), (6,3), (7,3), (12,4), (8,3), (9,3), (10,3)} {
\filldraw [fill=white] \x circle (5pt);
}
\foreach \x in {(2,3), (2,2), (3,2), (4,2), (5,2), (5,1), (6,1), (7,1),
(8,1), (8,0), (9,0), (10,0), (6,7), (6,6), (7,6), (7,5), (7,4),
(7,3), (7,2), (9,1), (10,1), (12,3)} {\fill \x circle (3pt);}
\draw [thick, rounded corners, rotate=45, shift={(3.2,.5)}] (0,0)
rectangle +(6.4,0.4);
\draw (12,4.1) node[right] {\, = tileable};
\draw (12,3.1) node[right] {\, = blocked};
\end{tikzpicture}
\end{center}
\caption{The path of tileable tiles from $u$ must intersect one of the paths
of blocked tiles from $s$ and $t$ to $w$, a contradiction.}
\label{F:rectPaths}
\end{figure}
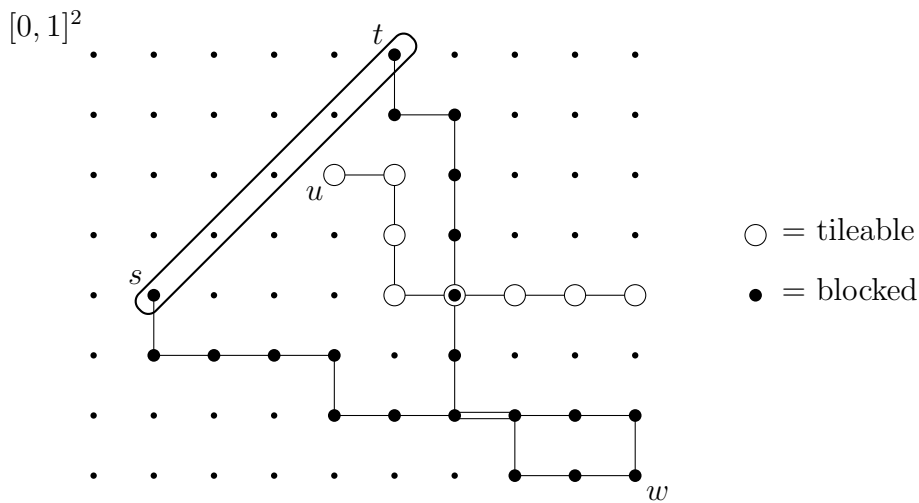

The two ancestries of blocked tiles from $w$ to each of $s$ and $t$
correspond to directed paths in the square lattice, as indicated in
Figure~\ref{F:rectPaths}.  (The two ancestries might {\em a priori}
intersect, although aside from $s$ and $t$ their tiles occur in chains at
distinct vertices in the chain tree).  Now consider the tile $u$, which abuts
some bond of the chain $[s,t]$ in the lattice as shown.  By
Lemma~\ref{L:trivial}, either both horizontal or both vertical children of
$u$ are tileable.  (This is the only place where we use the ``only if''
direction of Lemma~\ref{L:trivial}). Suppose that $u$ is strictly larger than
$w$ in both width and height.  Then $w$ is contained in one of $u$'s
horizontal children and one of its vertical children, hence $u$ has a child
that contains $w$ and is tileable.  Iterating this argument until we reach a
tile of the same width or height as $w$, we obtain a directed path of
tileable tiles in the lattice, starting at $u$ and ending in the row or
column containing $w$, as shown in Figure~\ref{F:rectPaths}.  Such a path
must intersect one of the two blocked ancestry paths, giving a contradiction.
\end{proof}

\begin{proof}[Proof of \cref{P:Tbound}]
Using \cref{L:principal,L:disjoint},
\begin{align*}
1 - T_n(p)
& = \P\bigl([0,1]^2\text{ is blocked}\bigr)\\
& = \P(\text{there exists a principal chain tree})\\
& \le \sum_T (1-p)^{t(T)} \mathbf{1}[\text{$T$ has pairwise disjoint chains}]\\
& \le \sum_T (1-p)^{t(T)} \\
& = f_n(1-p,1-p),
\end{align*}
where the sums are over all chain trees of depth $n$.
\end{proof}

\section{Analysis of the generating function}
\label{analysis}

To complete the proof of \cref{T:7/8}, all that remains is to
analyze the asymptotic behavior of $f_n$.  Recall that this is
the generating function for \textit{all} possible chain trees
of depth $n$---we will no longer be concerned with availability
or principal chain trees.

\begin{Prop}\label{P:f_rec}
The polynomials $f_n$ satisfy the recursion
\begin{equation} \label{f}
f_{n+1}(q,z) = f_n\bigl(q(1+z), 4qz\bigr).
\end{equation}
\end{Prop}

\begin{proof}
Given a chain tree $T$ of depth $n$, we may obtain a chain tree of depth
$n+1$ by choosing a successor of each leaf chain of $T$ and adding the
appropriate children.  All chain trees of depth $n+1$ can be obtained (each
exactly once) in this way.

Fix some chain of $\beta$ bonds, and for any successor $S$, write $b(S)$ for
the number of bonds and $c(S)$ for the number of pairwise separate chains.
Then by \cref{L:successor}, the generating function of the possible
successors is given by
\begin{equation}\label{one-ext}
\sum_S q^{b(S)} z^{c(S)} = \sum_{r=0}^\beta
\binom{\beta}{r}4q^{\beta+1}z^{r+1}=\big[q(1+z)\big]^\beta 4qz,
\end{equation}
where the sum is over all possible successors of the given chain.

Consider a depth-$n$ chain tree $T$ whose leaves have $c$ chains, $t$ tiles
and $b=t-c$ bonds (counted with multiplicities, as before).  This chain tree
contributes a term $q^b z^c$ to $f_n(q,z)$.  The possible extensions of $T$
to depth $n+1$ contribute various terms to $f_{n+1}(q,z)$, and since we may
choose any successor for each leaf independently of the others, the sum of
these terms is the product of expressions of the form \eqref{one-ext} over
the leaves of $T$.  That is:
\begin{equation}\label{one-tree}
\prod_L \bigl[q(1+z)\bigr]^{\beta(L)} 4qz = \bigl[q(1+z)\bigr]^{b} (4qz)^c,
\end{equation}
where the product is over the leaves of $T$, and $\beta(L)$ is the number of
bonds of the chain at leaf $L$.

The polynomial $f_{n+1}(q,z)$ is obtained by summing the expression in
\eqref{one-tree} over all chain trees of depth $n$.  Therefore, it is
obtained from $f_n(q,z)$ by replacing each term $q^b z^c$ with the term
$[q(1+z)]^b (4qz)^c$. This gives \eqref{f}.
\end{proof}

\begin{Cor}\label{L:q_rec}
Fix $q$ and $z$ and write $f_n=f_n(q,z)$.  For every $n$,
\begin{equation}\label{eq:q_rec}
\frac{f_{n+2}}{f_{n+1}} = \frac{f_{n+1}}{f_n} + f_{n+1}.
\end{equation}
\end{Cor}

\begin{proof}
For $n=0$, \eqref{eq:q_rec} is easily verified (recall that $f_0=z$,
$f_1=4qz$, and $f_2=16q^2z(1+z)$). Now the two-variable rational function
$R_n := f_{n+2}/f_{n+1} - f_{n+1}/f_n - f_{n+1}$ satisfies the same recursion
\eqref{f} as $f$, that is, $R_{n+1}(q,z) = R_n(q(1+z), 4qz)$; this is
immediate simply by substituting from \eqref{f} for each of
$f_{n+3},f_{n+2},f_{n+1}$.  It follows that $R_n$ is identically zero for
every $n$.
\end{proof}

Next, we show how to control the asymptotic behaviour of solutions to the
recursion \eqref{eq:q_rec}.

\begin{Lemma}\label{L:bounded}
  Suppose that $a_0,a_1>0$, and that the sequence $(a_n)$ satisfies the
  recursion \eqref{eq:q_rec}, in other words for every $n$,
  \begin{equation}\label{arec}
    \frac{a_{n+2}}{a_{n+1}} = \frac{a_{n+1}}{a_n} + a_{n+1}.
  \end{equation}
  If $X$ satisfies
  \begin{equation}\label{eq:X}
    \frac{a_k}{a_{k-1}} +  \frac{a_k}{1-X} \leq X < 1,
  \end{equation}
  for some $k$, then for every $n$ we have $a_{n+1}/a_n < X$ and so $a_n <
  a_0 X^n$.
\end{Lemma}

\begin{proof}
Write $Q_n=a_n / a_{n-1}$. Observe that $a_n>0$ for all $n$, by induction.
Therefore $Q_n$ is increasing in $n$. We need to prove that $Q_n<X$ for all
$n\geq 1$. By \eqref{eq:X} we have
$$Q_k < Q_k + \frac{a_k}{1-X} \leq X,$$
so the required inequality holds for $n=k$, and hence also for all $n<k$
since $Q_n$ is increasing.  For $n>k$ we use induction.  Suppose that $Q_n<X$
for all $n\leq m$, where $m\geq k$.   By repeated application of
\eqref{arec},
$$Q_{m+1} = Q_k + \sum_{i=k}^m a_i.$$
On the other hand the inductive hypothesis gives $a_i\leq a_k X^{i-k}$ for
$k\leq i\leq m$, so substituting into the last equation and using
\eqref{eq:X} gives
\[Q_{m+1} < Q_k +\frac{a_k}{1-X} \leq X.\qedhere \]
\end{proof}

\begin{Cor}\label{P:general}
  Let $0<q<1$, and define a sequence $(a_n)_{n\geq 0}$ by $a_0=q$,
  $a_1=4q^2$, and
  \[
  \frac{a_{n+2}}{a_{n+1}} = \frac{a_{n+1}}{a_n} + a_{n+1},\qquad n\geq 0
  \]
  If for some $k$ we have $a_k<a_{k-1}$ and
  \begin{equation}\label{ak}
    \Bigl(1-\frac{a_k}{a_{k-1}}\Bigr)^2 \geq 4 a_k,
  \end{equation}
  then for every $n$,
  \begin{equation} \label{Tgeneral}
    f_n(q,q) \leq \Bigl(\frac{1+a_k/a_{k-1}}2\Bigr)^n.
  \end{equation}
\end{Cor}

\begin{proof}
It is easy to verify that an $X$ satisfying the quadratic inequality
\eqref{eq:X} exists exactly when \eqref{ak} holds, in which case $X =
(1+a_k/a_{k-1})/2$ is a solution. Since $a_0=q<1$, \cref{L:bounded} implies
the result.
\end{proof}

\begin{proof}[Proof of \cref{T:7/8}]
We combine Corollary~\ref{P:general} with Proposition \ref{P:Tbound} for
suitable values of $k$ and $q:=1-p$.  With $q=1/8$ we get $a_0 = 1/8$ and
$a_1=1/16$. Therefore \eqref{ak} holds with equality for $k=1$, and
\eqref{Tgeneral} gives the first claim of \cref{T:7/8}.

Using arithmetic with rational numbers to avoid numerical errors, we have
verified that for $q=1/7$, \eqref{ak} is satisfied for $k=16$, and
$(1+a_{16}/a_{15})/2$ is slightly smaller than $16/17$. This establishes the
second claim of Theorem~\ref{T:7/8}, but the calculation involves integers
with more than $50000$ digits. Similarly, using interval arithmetic to avoid
errors gives the third claim.
\end{proof}

For the above values of $q$, the bounds on $T_n(p)$ are not the best that can
be obtained from our analysis.  Using larger $k$ and the optimal bound from
Lemma~\ref{L:bounded} yields slightly better exponential decay. For example,
this gives $T_n(7/8)\geq 1-0.655^n$. We believe that the correct exponential
decay is even faster, and that $p_c$ is strictly smaller than our bound.

\section*{Open problems}
\begin{enumerate}
\item[(i)] Does the tiling probability $T_n(p_c)$ {\em at} the critical
    point have a limit as $n\to\infty$?  If so, what is it?  Is it $1$?
    (As remarked in \cref{prelim}, such a limit must be at least $(\surd
    5-1)/2$.)
\item[(ii)] Is there a phase transition in dimensions $d\geq 4$?  I.e.,
    is $p_c(d)$ strictly positive?
\item[(iii)] Is there a unique critical point in dimension $3$?  I.e.,
    does the tiling probability $T_n^{(3)}(p)$ tend to $0$ or $1$ as
    $n\to\infty$ according as $p<p_c(3)$ or $p>p_c(3)$.  (As discussed in
    \cref{prelim,sec:sharp}, we know that there is a sharp transition for
    each $n$ in the sense that $T_n^{(3)}(p)$ increases from $\epsilon$
    to $1-\epsilon$ over an interval of length $o(1)$ as $n\to\infty$,
    and that the location of this interval is bounded strictly away from
    $0$ and $1$.  The question is whether the location converges or
    oscillates).
\end{enumerate}

\section*{Acknowledgements}
\enlargethispage*{1cm}
 This work arose from a meeting at the Banff
International Research Station (BIRS), Alberta, Canada.  We are grateful for
the use for this outstanding resource. We thank James Martin, Jim Propp, Dan
Romik and David Wilson for many valuable conversations.  We thank the referee
for helpful comments. Supported in part by NSERC (OA), the Israel Science
Foundation (GK), and NSF grant \#0901475 (PW).

\bibliographystyle{abbrv}
\bibliography{dyadic}

\end{document}